\documentclass[9pt]{article}%
\usepackage{bbm}
\usepackage{epsfig}
\usepackage{graphics,graphicx,amssymb,amsmath,verbatim}
\usepackage{amssymb}
\usepackage{mathrsfs}
\usepackage{amsfonts}
\usepackage{amsmath}
\usepackage{graphicx}
\usepackage{easybmat}%
\providecommand{\U}[1]{\protect \rule{.1in}{.1in}}
\newtheorem{theorem}{Theorem}

\newtheorem{corollary}{Corollary}

\newtheorem{example}{Example}

\newtheorem{lemma}{Lemma}

\newtheorem{remark}{Remark}

\newenvironment{proof}[1][Proof]{\noindent \textbf{#1.} }{\  \rule{0.5em}{0.5em}}
\begin{document}

\title{Positive Definite Solutions of the Nonlinear Matrix Equation $X+A^{\mathrm{H}%
}\overline{X}^{-1}A=I$}
\date{}
\author{Bin Zhou\thanks{Center for Control Theory and Guidance Technology, Harbin
Institute of Technology, P.O. Box 416, Harbin 150001, P. R. China.
E-mail: \texttt{binzhoulee@163.com, binzhou@hit.edu.cn.}}
\thanks{Corresponding author\texttt{.}}\qquad Guang-Bin
Cai\thanks{Unit 302, Department of Automation, Xi'an Research
Institute of High-Tech, Hongqing Town, Xi'an, 710025, Shaanxi, P. R.
China. E-mail: \texttt{cgb0712@163.com.}} \qquad James
Lam\thanks{James Lam is with Department of Mechanical Engineering,
University of Hong Kong, Hong Kong. E-mail:
\texttt{james.lam@hku.hk}.}} \maketitle

\begin{abstract}
This paper is concerned with the positive definite solutions to the
matrix equation $X+A^{\mathrm{H}}\overline{X}^{-1}A=I$ where $X$ is
the unknown and $A$ is a given complex matrix. By introducing and
studying a matrix operator on complex matrices, it is shown that the
existence of positive definite solutions of this class of nonlinear
matrix equations is equivalent to the existence of positive definite
solutions of the nonlinear matrix equation
$W+B^{\mathrm{T}}W^{-1}B=I$ which has been extensively studied in
the literature, where $B$ is a real matrix and is uniquely
determined by $A.$ It is also shown that if the considered nonlinear
matrix equation has a positive definite solution, then it has the
maximal and minimal solutions. Bounds of the positive definite
solutions are also established in terms of matrix $A$. Finally some
sufficient conditions and necessary conditions for the existence of
positive definite solutions of the equations are also proposed.


\end{abstract}
\section{Introduction}

Various kinds of matrix equations have received much attention in
the literature (see, for example, \cite{bhrh87ctmt},
\cite{ccl12amc}, \cite{chu89laa}, \cite{dh08amc}, \cite{dh10laa},
\cite{dh11amm}, \cite{dc05tac}, \cite{dld08amc}, \cite{fc05scl},
\cite{e93laa}, \cite{lw10ijcm}, \cite{lw10amc}, \cite{lhz09cma},
\cite{wlwc08cma}, \cite{zdl09scl}, \cite{zld09aml}, \cite{zd08scl},
and the references therein). Especially, the problem of finding
fixed points of the nonlinear matrix equation $X+A^{*}X^{-1}A=Q$
where $A$ and $Q>0$ are given and $X$ is unknown, has been
extensively studied in the last two decades. The interest of
studying such problem mainly relies on its applications in many
fields such as the analysis of ladder networks \cite{amt83cssp},
dynamic programming \cite{pw75rmp}, control theory \cite{gk85siam},
stochastic filtering \cite{akkw69ams} and statistics
\cite{ouellette81laa} (see \cite{jr90laa}, \cite{err93laa} and
\cite{e93laa} for detailed introduction). Some generalized forms of
the nonlinear matrix equation $X+A^{*}X^{-1}A=Q$ have received much
attention in recent years (see, for example, \cite{cc10amc},
\cite{dll11amc}, \cite{jw09amc}, and \cite{lyd10amc}).

Among the existing publications in the literature, two kinds of
results can be found. The first kind of results concentrate on
providing analytical conditions on the existence of positive
definite solutions and their corresponding properties. For example,
the shorted operator theory was applied in \cite{jr90laa} to study
the existence of a positive definite solution; necessary and
sufficient conditions in terms of symmetric factorizations of some
rational matrix-valued function were derived in \cite{err93laa} for
the existence of positive definite solutions; and some necessary and
sufficient conditions were also derived \cite{zx96laa} for the same
problem in terms of some factorizations of the coefficient matrices.
The other kind of results are mainly concerned with the numerical
solutions of this class of nonlinear matrix equations. Basically,
this can be accomplished via iterations including inversion-involved
iterations \cite{e93laa}, \cite{gl99mc} and inversion-free
iterations \cite{elsayed03anziam}, \cite{ihu04mc}, \cite{mr10laa},
\cite{zhan95siam}.

In the present paper, we consider a variation of this well-studied
nonlinear equation. We study the nonlinear matrix equation
$X+A^{*}\overline {X}^{-1}A=Q$, which, as we will show in this
paper, has totally different solutions from the solutions of
$X+A^{*}X^{-1}A=Q.$ In particular, we are interested in the
existence of positive definite solutions of such kind of nonlinear
matrix equations. Via some specific representations of complex
matrices, we are able to transform the equation $X+A^{*}\overline
{X}^{-1}A=Q$ into the equation $W+B^{*}W^{-1}B=P$, where $B$ and $P$
are determined by $A$ and $Q$, respectively. This allows us to study
the original nonlinear matrix equations with the help of the
existing results on the equation $X+A^{*}X^{-1}A=Q.$ Other topics of
this paper include the estimate of the bounds on the solutions,
sufficient conditions, and necessary conditions on guaranteeing a
positive definite solution.

The rest of this paper is organized as follows. The problem
formulation and some preliminary results to be used are given in
Section \ref{sec2}. In Section \ref{sec3}, we present necessary and
sufficient conditions for the existence of a positive definite
solution of the considered nonlinear matrix equations. Both upper
bounds and lower bounds of the solutions will be established in
Section \ref{sec4}, while the necessary conditions and sufficient
conditions guaranteeing a positive definite solution are given in
Section \ref{sec5}. We will draw the conclusions of this paper in
Section \ref{sec6}.

\textbf{Notation}: In this paper, for a matrix $A$, we use $A^{\mathrm{T}%
},A^{*},\overline{A},$ $\lambda \left(  A\right)  ,\det \left(
A\right) $ $,\left \Vert A\right \Vert $ and $\rho \left(  A\right)
$ to denote respectively the transpose, the conjugated transpose,
the conjugate, the spectrum, the determinant, the $2$-norm, and the
spectral radius of $A.$ Moreover, $\omega \left(  A\right)  =\max
\left \{  \left \vert z\right \vert :z=x^{*}Ax,\left \Vert x\right
\Vert =1\right \} $ is the numerical radius of $A.$ Finally, the
symbol $P>0$ means that $P$ is positive definite, $I_{n}$ denotes an
$n\times n$ identity matrix, $\mathbf{0}$ denotes a zero matrix with
appropriate dimensions, and $\mathrm{j}=\sqrt{-1}$.

\section{\label{sec2}Problem Formulation and Preliminary Results}

We consider the following nonlinear matrix equation%
\begin{equation}
X+A^{*}\overline{X}^{-1}A=Q \label{eq1}%
\end{equation}
where $Q\in \mathbf{C}^{n\times n}$ is a given positive definite
matrix, $A\in \mathbf{C}^{n\times n}$ is a given complex matrix, and
$X\in \mathbf{C}^{n\times n}$ is the unknown. In this paper, we are
interested in the existence of positive definite solutions of this
class of nonlinear matrix equations.

\begin{remark}
Similar to \cite{zld11laa}, we can also consider positive definite
solutions
of matrix equation%
\begin{equation}
X+A^{*}X^{-\mathrm{T}}A=Q. \label{eq50}%
\end{equation}
However, the positive definite solutions of (\ref{eq50}) coincide
with the
positive definite solutions of equation (\ref{eq1}) since $X^{-\mathrm{T}%
}=\overline{X}^{-*}=\overline{X}^{-1}.$
\end{remark}

Via a simple manipulation we can show the following result.

\begin{lemma}
\label{lm6}Let $Q$ be a positive definite matrix. Then $X$ is a
solution of (\ref{eq1}) if and only if
$Y=Q^{-\frac{1}{2}}XQ^{-\frac{1}{2}}$ is a solution
of the following nonlinear matrix equation%
\[
I_{n}=Y+A_{Q}^{*}\overline{Y}^{-1}A_{Q}, \quad A_{Q}%
=\overline{Q}^{-\frac{1}{2}}AQ^{-\frac{1}{2}}.
\]

\end{lemma}

Therefore, without loss of generality, we assume hereafter that
$Q=I_{n}$ in (\ref{eq1})$.$ We point out that matrix equation
(\ref{eq1}) has solutions that are totally different from solutions
of the following nonlinear matrix
equation%
\begin{equation}
X+A^{*}X^{-1}A=I_{n}. \label{eq100}%
\end{equation}
See the following example for illustration.

\begin{example}
\label{example1}Consider a nonlinear matrix equation in the form of
(\ref{eq1}) with $Q=I_{2}$ and%
\[
A=\left[
\begin{array}
[c]{cc}%
\frac{1}{4}+\frac{1}{4}\mathrm{j} & \frac{1}{4}\mathrm{j}\\
-\frac{1}{4}\mathrm{j} & \frac{1}{4}-\frac{1}{4}\mathrm{j}%
\end{array}
\right]  .
\]
Then according to the results we will give later, we find the
maximal positive
definite solution of this equation as%
\[
X_{+}=\left[
\begin{array}
[c]{cc}%
\frac{1}{2}+\frac{1}{8}\sqrt{6} & -\frac{1}{8}-\frac{1}{8}\mathrm{j}\\
-\frac{1}{8}+\frac{1}{8}\mathrm{j} & \frac{1}{2}+\frac{1}{8}\sqrt{6}%
\end{array}
\right]  .
\]
However, according to the results in \cite{err93laa}, the maximal
positive
definite solution of equation (\ref{eq100}) can be computed as%
\[
X_{+}^{\prime}=\left[
\begin{array}
[c]{cc}%
\frac{1}{8}\sqrt{2}+\frac{1}{2} & -\frac{1}{4}-\frac{1}{8}\sqrt{2}\mathrm{j}\\
-\frac{1}{4}+\frac{1}{8}\sqrt{2}\mathrm{j} & \frac{1}{8}\sqrt{2}+\frac{1}{2}%
\end{array}
\right]  .
\]
It is clearly that $X_{+}\neq X_{+}^{\prime}.$
\end{example}

For a complex matrix $A=A_{1}+A_{2}\mathrm{j}\in \mathbf{C}^{n\times
m}$ where $A_{1},A_{2}\in \mathbf{R}^{n\times m},$ we denote the
operators $\left( \cdot \right)  ^{\heartsuit}$ and $\left(  \cdot
\right)  ^{\lozenge}$ as
\[
A^{\heartsuit}=\left[
\begin{array}
[c]{cc}%
A_{1} & -A_{2}\\
A_{2} & A_{1}%
\end{array}
\right]  ,\quad A^{\lozenge}=\left[
\begin{array}
[c]{cc}%
A_{2} & A_{1}\\
A_{1} & -A_{2}%
\end{array}
\right]  .
\]
It follows that both $A^{\heartsuit}$ and $A^{\lozenge}$ are real
matrices.
For further use, we define two unitary matrices $E_{n}$ and $P_{n}$ as%

\begin{equation}
E_{n}=\left[
\begin{array}
[c]{cc}%
0 & I_{n}\\
I_{n} & 0
\end{array}
\right]  ,\quad P_{n}=\frac{\sqrt{2}}{2}\left[
\begin{array}
[c]{cc}%
\mathrm{j}I_{n} & I_{n}\\
I_{n} & \mathrm{j}I_{n}%
\end{array}
\right]  . \label{eq20}%
\end{equation}
Some basic properties of these matrices are collected as the
following lemma whose proof is provided in Appendix B.

\begin{lemma}
\label{lm0}Let $A\in \mathbf{C}^{n\times m}$ and $B\in
\mathbf{C}^{m\times p}$ be two given complex matrices.

\begin{enumerate}
\item The following equalities are true.%
\[%
\begin{array}
[c]{lll}%
\left(  AB\right)  ^{\heartsuit}=A^{\heartsuit}B^{\heartsuit}, &
\left( A^{-1}\right)  ^{\heartsuit}=\left(  A^{\heartsuit}\right)
^{-1}, & \left( A^{\mathrm{T}}\right)  ^{\heartsuit}=E_{m}\left(
A^{\heartsuit}\right)
^{\mathrm{T}}E_{n},\\
\left(  A^{*}\right)  ^{\heartsuit}=\left( A^{\heartsuit}\right)
^{\mathrm{T}}, & \left(  \overline{A}\right)  ^{\heartsuit}=E_{n}%
A^{\heartsuit}E_{m}, & A^{\diamondsuit}=E_{n}A^{\heartsuit}.
\end{array}
\]

\item Let $P_{n}$ be defined in (\ref{eq20}). Then
\begin{equation}
A^{\heartsuit}=P_{n}\left[
\begin{array}
[c]{cc}%
A & 0\\
0 & \overline{A}%
\end{array}
\right]  P_{m}^{*}. \label{teq65}%
\end{equation}
Consequently, $A\geq \left(  >\right)  0$ if and only if $A^{\heartsuit}%
\geq \left(  >\right)  0.$

\item $A\in \mathbf{C}^{n\times n}$ is a normal (unitary) matrix if and only if
$A^{\heartsuit}$ is a real normal (unitary) matrix.

\item For any $A\in \mathbf{C}^{n\times n},$ there holds $\rho \left(  A\right)
=\rho \left(  A^{\heartsuit}\right)  $ and $\rho \left(
A^{\lozenge}\right) =\rho^{\frac{1}{2}}\left(  A\overline{A}\right)
$.

\item The real matrix $A^{\lozenge}$ is normal if and only if $A$ is
con-normal, namely, $A^{*}A=\overline{AA^{*}}.$

\item For any $A\in \mathbf{C}^{n\times m},$ there holds $\left \Vert
A^{\lozenge}\right \Vert =\left \Vert A^{\heartsuit}\right \Vert
=\left \Vert A\right \Vert $

\item If $A\geq0$, then $\left(  A^{\heartsuit}\right)  ^{\frac{1}{2}}=\left(
A^{\frac{1}{2}}\right)  ^{\heartsuit}.$
\end{enumerate}
\end{lemma}

Two matrices $A,B\in \mathbf{C}^{n\times n}$ are said to be
con-similar if there exists a nonsingular matrix $S\in
\mathbf{C}^{n\times n}$ such that $S^{-1}B\overline{S}=A.$ The
following lemma is borrowed from \cite{bhrh87ctmt}.

\begin{lemma}
\label{lm5}Let $J_{k}\left(  \lambda \right)  $ denote a $k\times k$
Jordan matrix whose diagonal elements are $\lambda.$ Then any matrix
$A\in \mathbf{C}^{n\times n}$ is con-similar to a direct sum of
blocks of the form
$J_{k}\left(  \lambda \right)  $ where $\lambda \geq0$ or%
\[
\left[
\begin{array}
[c]{cc}%
0 & I_{k}\\
J_{k}\left(  \lambda \right)  & 0
\end{array}
\right]
\]
where $\lambda<0$ or $\\mathrm{Im}\left(  \lambda \right) \neq0.$
\end{lemma}

The set of complex numbers $\lambda$ appearing in the $J_{k}\left(
\lambda \right)  $ in the blocks of the the canonical form defined
in Lemma \ref{lm5} of a matrix $A\ $will be called the con-spectrum
of $A$ and denoted by $\mathrm{co}\lambda \left(  A\right)  $
\cite{bhrh87ctmt}. Moreover, the con-spectrum-radius of $A$ will be
denoted by
\[
\mathrm{co}\rho \left(  A\right)  =\max \{ \left \vert \lambda
\right \vert :\lambda \in \mathrm{co}\lambda \left(  A\right)  \}.
\]
A property of the\ con-spectrum-radius is given in the following
lemma whose proof will be presented in Appendix C.

\begin{lemma}
\label{lm2}Let $A\in \mathbf{C}^{n\times n}$ be a given matrix. Then
$\mathrm{co}\rho \left(  A\right)  \leq \left(  \geq \right)
1\Leftrightarrow \rho \left(  A\overline{A}\right)  \leq \left(
\geq \right)  1.$
\end{lemma}

At the end of this section, we recall the well-known Schur
complement.

\begin{lemma}
Let matrix $\mathit{\Phi}$ be defined as%
\[
\mathit{\Phi}=\left[
\begin{array}
[c]{cc}%
\mathit{\Phi}_{11} & \mathit{\Phi}_{12}\\
\mathit{\Phi}_{12}^{*} & \mathit{\Phi}_{22}%
\end{array}
\right]  .
\]
Then the following three statements are equivalent:

\begin{enumerate}
\item $\mathit{\Phi}>0$.

\item $\mathit{\Phi}_{11}>0$ and $\mathit{\Phi}_{22}-\mathit{\Phi}%
_{12}^{*}\mathit{\Phi}_{11}^{-1}\mathit{\Phi}_{12}>0.$

\item $\mathit{\Phi}_{22}>0$ and $\mathit{\Phi}_{11}-\mathit{\Phi}%
_{12}\mathit{\Phi}_{22}^{-1}\mathit{\Phi}_{12}^{*}>0.$
\end{enumerate}
\end{lemma}

\section{\label{sec3}Necessary and Sufficient Conditions}

In this section, we study necessary and sufficient conditions for
the existence of a positive definite solution of the nonlinear
matrix equation (\ref{eq1}). Firstly we introduce an useful lemma.

\begin{lemma}
\label{lm1}Assume that $A$ is nonsingular. Then $X$ solves the
nonlinear matrix equation (\ref{eq1}) if and only if
$Y=I_{n}-\overline{X}$ solves the
following nonlinear matrix equation%
\begin{equation}
I_{n}=Y+A\overline{Y}^{-1}A^{*}. \label{eq5}%
\end{equation}

\end{lemma}

\begin{proof}
Let $X$ be a solution of equation (\ref{eq1}), then $A^{*}%
\overline{X}^{-1}A=I_{n}-X$ from which we get $\overline{X}^{-1}%
=A^{-*}\left(  I_{n}-X\right)  A^{-1}.$ Taking inverses on both
sides gives $\overline{X}=A\left(  I_{n}-X\right) ^{-1}A^{*}$ which
is equivalent to equation (\ref{eq5}) by setting
$I_{n}-\overline{X}=Y.$ The converse can be shown similarly.
\end{proof}

Our main result regarding the existence of positive definite
solution of equation (\ref{eq1}) is presented as follows.

\begin{theorem}
\label{th1}The nonlinear matrix equation in (\ref{eq1}) has a
solution $X>0$
if and only if the following nonlinear matrix equation%
\begin{equation}
I_{2n}=W+\left(  A^{\lozenge}\right)
^{\mathrm{T}}W^{-1}A^{\lozenge},
\label{eq2}%
\end{equation}
has a solution $W>0.$ Moreover, the following two statements hold
true:

\begin{enumerate}
\item If the nonlinear matrix equation in (\ref{eq1}) has a solution $X>0,$
then it must have a maximal positive definite solution $X_{+}.$
Particularly, if $W_{+}$ denotes the maximal solution of the
nonlinear matrix equation in (\ref{eq2}), then
$W_{+}=X_{+}^{\heartsuit}$, or
\begin{equation}
X_{+}=\frac{1}{2}\left[
\begin{array}
[c]{c}%
\mathrm{j}I_{n}\\
I_{n}%
\end{array}
\right]  ^{*}W_{+}\left[
\begin{array}
[c]{c}%
\mathrm{j}I_{n}\\
I_{n}%
\end{array}
\right]  . \label{eq6}%
\end{equation}

\item If $A$ is nonsingular and the nonlinear matrix equation (\ref{eq1}) has
positive definite solution $X>0$, then it must have a minimal
positive definite solution $X_{-}.$ Particularly, if $W_{-}$ denotes
the minimal
solution of the nonlinear matrix equation in (\ref{eq2}), then $W_{-}%
=X_{-}^{\heartsuit}$, or
\begin{equation}
X_{-}=\frac{1}{2}\left[
\begin{array}
[c]{c}%
\mathrm{j}I_{n}\\
I_{n}%
\end{array}
\right]  ^{*}W_{-}\left[
\begin{array}
[c]{c}%
\mathrm{j}I_{n}\\
I_{n}%
\end{array}
\right]  . \label{eq7}%
\end{equation}

\end{enumerate}
\end{theorem}

\begin{proof}
\textquotedblleft$\Longrightarrow$\textquotedblright \ Let $X>0$ be
a solution of equation (\ref{eq1}). Taking $\left(  \cdot \right)
^{\heartsuit}$ on both sides of equation (\ref{eq1}) and using Lemma
\ref{lm0} gives
\begin{align}
I_{n}^{\heartsuit}  &  =X^{\heartsuit}+\left(  A^{*}\overline{X}%
^{-1}A\right)  ^{\heartsuit}\nonumber \\
&  =X^{\heartsuit}+\left(  A^{\heartsuit}\right)
^{\mathrm{T}}\left(  \left(
\overline{X}\right)  ^{\heartsuit}\right)  ^{-1}A^{\heartsuit}\nonumber \\
&  =X^{\heartsuit}+\left(  A^{\heartsuit}\right)
^{\mathrm{T}}\left(
E_{n}X^{\heartsuit}E_{n}^{\mathrm{T}}\right)  ^{-1}A^{\heartsuit}\nonumber \\
&  =X^{\heartsuit}+\left(  A^{\lozenge}\right)  ^{\mathrm{T}}\left(
X^{\heartsuit}\right)  ^{-1}A^{\lozenge}, \label{eq13}%
\end{align}
which indicates that $W=X^{\heartsuit}>0$ is a solution of equation
(\ref{eq2}).

\textquotedblleft$\Longleftarrow$\textquotedblright \ Let equation
(\ref{eq2}) have a solution $W>0.$ Then it must have a maximal
solution according to Lemma \ref{lm8} in Appendix A. We denote such
maximal solution by $W_{+}.$ In the following we show that there
must exist a matrix $Y>0$ such that $W_{+}=Y^{\heartsuit}.$
According to Lemma \ref{lm9} in Appendix A, we know that
\begin{equation}
W_{k+1}=I_{2n}-\left(  A^{\lozenge}\right)  ^{\mathrm{T}}W_{k}^{-1}%
A^{\lozenge},\quad W_{0}=I_{2n}, \label{eq22}%
\end{equation}
converges monotonically to $W_{+},$ namely, $0<W_{+}\leq W_{k+1}\leq
W_{k}\leq I_{2n},k\geq0$ and
\begin{equation}
\lim_{k\rightarrow \infty}W_{k}=W_{+}>0. \label{eq11}%
\end{equation}
We show that, for any integer $k\geq0,$ there exists a matrix
$Y_{k}>0$ such that
\begin{equation}
W_{k}=Y_{k}^{\heartsuit},\quad \forall k\geq0. \label{eq12}%
\end{equation}
We show this by induction. Clearly, equation (\ref{eq12}) holds true
for $k=0$ by setting $Y_{0}=I_{n}.$ Assume that (\ref{eq12}) is true
with $k=s,$ say, there exists a $Y_{s}>0$ such that
$W_{s}=Y_{s}^{\heartsuit}.$ Then, for
$k=s+1,$ by applying Lemma \ref{lm0}, we have%
\begin{align*}
W_{s+1}  &  =I_{2n}-\left(  A^{\lozenge}\right)  ^{\mathrm{T}}\left(
Y_{s}^{\heartsuit}\right)  ^{-1}A^{\lozenge}\\
&  =I_{2n}-\left(  A^{\heartsuit}\right)  ^{\mathrm{T}}\left(  E_{n}%
Y_{s}^{\heartsuit}E_{n}^{\mathrm{T}}\right)  ^{-1}A^{\heartsuit}\\
&  =I_{2n}-\left(  A^{*}\right)  ^{\heartsuit}\left( \overline
{Y_{s}}^{-1}\right)  ^{\heartsuit}A^{\heartsuit}\\
&  =I_{2n}-\left(  A^{*}\overline{Y_{s}}^{-1}A\right) ^{\heartsuit
}\\
&  =Y_{s+1}^{\heartsuit},
\end{align*}
where $Y_{s+1}=I-A^{*}\overline{Y_{s}}^{-1}A.$ As $W_{s+1}>0,$ we
know that $Y_{s+1}>0$ also. Therefore, (\ref{eq12}) is proved by
induction.

Hence, it follows from (\ref{eq11}) and (\ref{eq12}) that there
exists a
matrix $Y_{\infty}>0$ such that%
\[
W_{+}=\lim_{k\rightarrow \infty}W_{k}=\lim_{k\rightarrow \infty}Y_{k}%
^{\heartsuit}=Y_{\infty}^{\heartsuit},
\]
namely,%
\[
I_{2n}=Y_{\infty}^{\heartsuit}+\left(  A^{\lozenge}\right)  ^{\mathrm{T}%
}\left(  Y_{\infty}^{\heartsuit}\right)  ^{-1}A^{\lozenge},
\]
which, by using a similar technique used in deriving (\ref{eq13}),
is
equivalent to%
\[
I_{n}=Y_{\infty}+A^{*}\overline{Y}_{\infty}^{-1}A.
\]
Hence the nonlinear matrix equation (\ref{eq1}) has a solution
$X=Y_{\infty}.$

\textit{Proof of Item 1:} Since for any positive definite solution
$X$ of (\ref{eq1})$,$ there is a real positive definite solution
$X^{\heartsuit}$ of (\ref{eq2}), we must have $X^{\heartsuit}\leq
W_{+}=Y_{\infty}^{\heartsuit}$ which indicates that $X\leq
Y_{\infty}.$ However, $X=Y_{\infty}$ is also a solution of the
nonlinear matrix equation (\ref{eq1}). Hence, the nonlinear matrix
equation (\ref{eq1}) has the maximal solution $X_{+}=Y_{\infty}.$
The relation $W_{+}=X_{+}^{\heartsuit}$ then follows directly.

\textit{Proof of Item 2}:\quad If the nonlinear matrix equation
(\ref{eq1}) has a positive definite solution $X>0$, then, as
$X<I_{n},$ by Lemma \ref{lm1}, the nonlinear matrix equation
(\ref{eq5}) also has a positive definite solution
$I_{n}-\overline{X}$, which, according to item 1 of this theorem,
indicates that equation (\ref{eq5}) must have a maximal positive
definite solution $Y_{+}.$ Hence, by applying Lemma \ref{lm1} again,
$X_{-}\triangleq I_{n}-\overline{Y}_{+}$ is also a solution of
equation (\ref{eq1}). In fact, $X_{-}$ is the minimal positive
definite solution of equation (\ref{eq1}). Otherwise assume that
$X_{\ast}\leq X_{-}$ is a positive definite solution of (\ref{eq1}).
Then
\[
Y_{\ast}=I_{n}-\overline{X}_{\ast}\geq I_{n}-\overline{X}_{-}=Y_{+},
\]
is a positive definite solution of (\ref{eq5}). This is impossible
since $Y_{+}$ is the maximal positive definite solution of equation
(\ref{eq5}) by assumption.

According to item 1 of this theorem, the maximal positive definite
solution $Y_{+}$ to equation (\ref{eq5}) is related with
$Z_{+}=Y_{+}^{\heartsuit}$
where $Z_{+}$ is the maximal positive definite solution of%
\begin{equation}
I_{2n}=Z+A^{\lozenge}Z^{-1}\left(  A^{\lozenge}\right)
^{\mathrm{T}}.
\label{eq8}%
\end{equation}
As $A^{\lozenge}$ is nonsingular, by applying Lemma \ref{lm1} again,
the maximal positive definite solution $Z_{+}$ to equation
(\ref{eq8}) is related with $Z_{+}=I_{2n}-Y_{-}$ where $Y_{-}$ is
the minimal positive definite solution of equation (\ref{eq2}).
Hence the minimal positive definite solution
$X_{-}$ to equation (\ref{eq1}) satisfies%
\[
X_{-}^{\heartsuit}=\left(  I_{n}-Y_{+}\right)  ^{\heartsuit}=I_{2n}%
-Z_{+}=Y_{-}.
\]
The proof is finished.
\end{proof}

\begin{remark}
The nonlinear matrix equation in (\ref{eq2}) is in the standard from
of (\ref{EqNonlinear}) with $Q=I_{n}$ (see Appendix) which has been
extensively studied in the literature. Therefore, by adopting the
existing results on the nonlinear matrix equation (\ref{eq2}), we
can get corresponding results on the original nonlinear matrix
equation (\ref{eq1}).
\end{remark}

According to the proof of Theorem \ref{th1}, we have the following
result regarding iteration based numerical solution of the nonlinear
matrix equation (\ref{eq1}).

\begin{corollary}
\label{coro2}Assume that the nonlinear matrix equation (\ref{eq1})
has a positive definite solution. Denote the largest solution by
$X_{+}.$ Then the
iteration%
\begin{equation}
W_{k+1}=I_{2n}-\left(  A^{\lozenge}\right)  ^{\mathrm{T}}W_{k}^{-1}%
A^{\lozenge},\quad W_{0}=I_{2n}, \label{eq14}%
\end{equation}
converges and such that $\lim_{k\rightarrow
\infty}W_{k}=X_{+}^{\heartsuit}.$ Moreover, if $\left \Vert
X_{+}^{-1}\overline{A}\right \Vert <1$, then the iteration in
(\ref{eq14}) converges to $X_{+}^{\heartsuit}$ with at least a
linear convergence rate, namely, there exists a $k^{\ast}>0$ and a
number
$0<\mu<1$ such that%
\begin{equation}
\left \Vert W_{k+1}-X_{+}^{\heartsuit}\right \Vert \leq \mu \left \Vert W_{k}%
-X_{+}^{\heartsuit}\right \Vert ,\quad \forall k\geq k^{\ast}.
\label{eq15}%
\end{equation}

\end{corollary}

\begin{proof}
The convergence of the iteration (\ref{eq14}) follows from the proof
of Theorem \ref{th1}. So we need only to show (\ref{eq15}). We use
the idea found in \cite{zhan95siam} to prove the result. According
to Theorem \ref{th1}, we have $X_{+}^{\heartsuit}=W_{+},$ the
maximal solution to equation (\ref{eq2}).
Notice that%
\begin{align}
\left \Vert W_{k+1}-X_{+}^{\heartsuit}\right \Vert  &  =\left \Vert W_{k+1}%
-W_{+}\right \Vert \nonumber \\
&  =\left \Vert I_{2n}-\left(  A^{\lozenge}\right)  ^{\mathrm{T}}W_{k}%
^{-1}A^{\lozenge}-\left(  I_{2n}-\left(  A^{\lozenge}\right)  ^{\mathrm{T}%
}W_{+}^{-1}A^{\lozenge}\right)  \right \Vert \nonumber \\
&  =\left \Vert \left(  A^{\lozenge}\right)  ^{\mathrm{T}}\left(  W_{+}%
^{-1}-W_{k}^{-1}\right)  A^{\lozenge}\right \Vert \nonumber \\
&  =\left \Vert \left(  W_{+}^{-1}A^{\lozenge}\right)
^{\mathrm{T}}\left(
W_{k}-W_{+}\right)  W_{k}^{-1}A^{\lozenge}\right \Vert \nonumber \\
&  \leq \left \Vert W_{+}^{-1}A^{\lozenge}\right \Vert \left \Vert W_{k}%
^{-1}A^{\lozenge}\right \Vert \left \Vert W_{k}-W_{+}\right \Vert . \label{eq16}%
\end{align}
On the other hand, by using Lemma \ref{lm0}, we can compute%
\begin{align*}
\left \Vert W_{+}^{-1}A^{\lozenge}\right \Vert  &  =\left \Vert
\left( X_{+}^{-1}\right)  ^{\heartsuit}E_{n}A^{\heartsuit}\right
\Vert =\left \Vert
E_{n}\left(  X_{+}^{-1}\right)  ^{\heartsuit}E_{n}A^{\heartsuit}\right \Vert \\
&  =\left \Vert \left(  \overline{X}_{+}^{-1}\right)  ^{\heartsuit
}A^{\heartsuit}\right \Vert =\left \Vert \left(
\overline{X}_{+}^{-1}A\right)
^{\heartsuit}\right \Vert \\
&  =\left \Vert \overline{X}_{+}^{-1}A\right \Vert =\left \Vert X_{+}%
^{-1}\overline{A}\right \Vert .
\end{align*}
Let $\mu \in \left(  0,1\right)  $ be such that $\left \Vert W_{+}^{-1}%
A^{\lozenge}\right \Vert =\left \Vert X_{+}^{-1}\overline{A}\right
\Vert <\sqrt{\mu}<1.$ As $\lim_{k\rightarrow
\infty}W_{k}=X_{+}^{\heartsuit},$ there exists a $k^{\ast}$ such
that $\left \Vert W_{k}^{-1}A^{\lozenge}\right \Vert
<\sqrt{\mu}<1,\forall k\geq k^{\ast}.$ Hence, the inequality in
(\ref{eq16}) reduces to (\ref{eq15}) immediately. The proof is
finished.
\end{proof}

We emphasize that the condition $\left \Vert
X_{+}^{-1}\overline{A}\right \Vert <1$ is only sufficient for
guaranteeing the linear convergence of the iteration in (\ref{eq14})
which converges as long as the nonlinear matrix equation (\ref{eq1})
has a positive definite solution. By combining Lemma \ref{lm1} and
Corollary \ref{coro2}, we can also present a result regarding
obtaining the minimal solution to the nonlinear matrix equation
(\ref{eq1}). The details are omitted for brevity.

\begin{example}
Consider the nonlinear matrix equation in Example \ref{example1}. By
computation we have $\left \Vert X_{+}^{-1}\overline{A}\right \Vert
=0.614<1.$ Then by Corollary \ref{coro2}, we conclude that the
corresponding iteration in (\ref{eq14})\ converges to
$X_{+}^{\heartsuit}$ at least linearly. For illustration, the
history of the iteration is recorded in Figure \ref{fig1}. From this
figure we see that the convergence of the corresponding iteration in
(\ref{eq14}) is indeed linear. Hence, the estimation in (\ref{eq15})
may be nonconservative.
\end{example}

\begin{figure}[ptb]
\begin{center}
\includegraphics[scale=0.8]{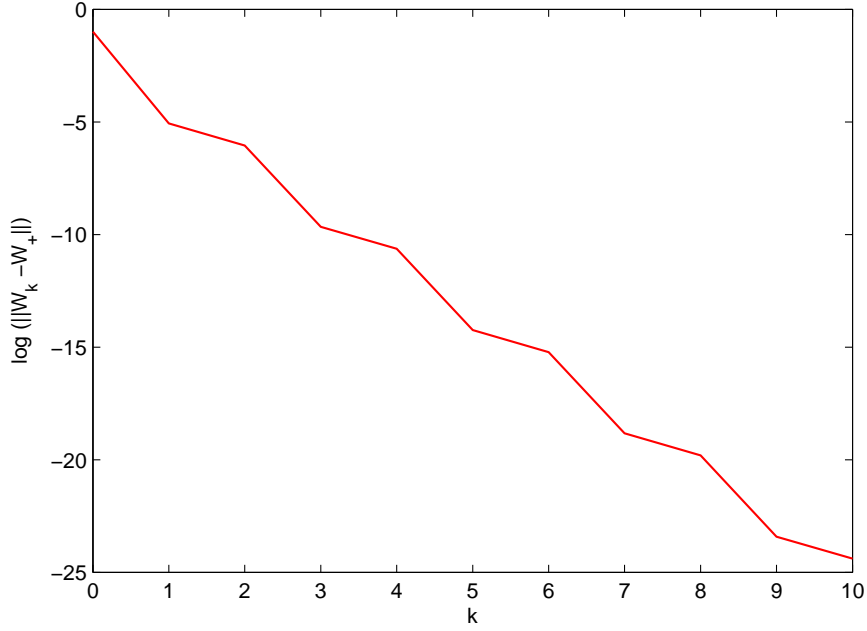}
\end{center}
\caption{Numerical solution of the nonlinear matrix equation
(\ref{eq1}) via
iteration (\ref{eq14})}%
\label{fig1}%
\end{figure}

In the particular case that $A$ is real, we can show the following
result.

\begin{corollary}
\label{coro1}Suppose that $A$ is real. If the nonlinear matrix
equation (\ref{EqNonlinear}) has a positive definite solution, then
$X_{+}$ is real. Furthermore, if $A$ is nonsingular, then $X_{-}$ is
also real. Hence in this
case, equation (\ref{eq1}) and the following nonlinear matrix equation%
\[
I_{n}=X+A^{\mathrm{T}}X^{-1}A,
\]
has the same maximal and minimal positive definite solutions.
\end{corollary}

\begin{proof}
Under the assumption of this corollary, $W_{+}=X_{+}^{\heartsuit}$
and $W_{-}=X_{-}^{\heartsuit}$ are respectively the maximal and
minimal solution of equation (\ref{eq2}). Therefore, we need only to
show that $W_{+}$ and
$W_{-}$ are in the form of%
\[
W_{+}=\left[
\begin{array}
[c]{cc}%
W_{1+} & 0\\
0 & W_{1+}%
\end{array}
\right]  ,\quad W_{-}=\left[
\begin{array}
[c]{cc}%
W_{1-} & 0\\
0 & W_{1-}%
\end{array}
\right]  .
\]
Since $W_{+}$ is the limit of the iteration in (\ref{eq22}), we only
need to
show that%
\begin{equation}
W_{k}=\left[
\begin{array}
[c]{cc}%
W_{1k} & 0\\
0 & W_{1k}%
\end{array}
\right]  ,\quad \forall k\geq0. \label{eq23}%
\end{equation}
We show this by induction. Clearly, (\ref{eq23}) holds true for
$k=0.$ Assume
that it is true with $k=s.$ Then, for $k=s+1,$ we can compute%
\[
W_{s+1}=\left[
\begin{array}
[c]{cc}%
I_{n}-A_{1}^{\mathrm{T}}W_{1s}A_{1} & 0\\
0 & I_{n}-A_{1}^{\mathrm{T}}W_{1s}A_{1}%
\end{array}
\right]  \triangleq \left[
\begin{array}
[c]{cc}%
W_{1\left(  s+1\right)  } & 0\\
0 & W_{1\left(  s+1\right)  }%
\end{array}
\right]  .
\]
Therefore, (\ref{eq23}) is proved by induction. The case $W_{-}$ can
be proved similarly.
\end{proof}

Combining Theorem \ref{th1}\ and Lemma \ref{lm10} gives the
following corollary.

\begin{corollary}
Suppose that $A$ is invertible. Then the nonlinear matrix equation
(\ref{eq1}) has a positive definite solution if and only if $\omega
\left(  A^{\lozenge }\right)  \leq \frac{1}{2}.$
\end{corollary}

Our next theorem presents some properties of the maximal and minimal
positive definite solutions of the nonlinear matrix equation
(\ref{eq1}).

\begin{theorem}
\label{th2}Assume that the nonlinear matrix equation in (\ref{eq1})
has a solution $X>0.$

\begin{enumerate}
\item Let the maximal positive definite solution be $X_{+}.$ Then $X_{+}$ is
the unique solution such that
$X_{+}+s\overline{A}\overline{X}_{+}^{-1}A$ is invertible for all
$s\in \{z:\left \vert z\right \vert <1\}$, or equivalently,
\begin{equation}
\mathrm{co}\rho \left(  \overline{X}_{+}^{-1}A\right)  \leq1. \label{eq9}%
\end{equation}

\item Assume further that $A$ is nonsingular. Let the minimal positive
definite solution be $X_{-}.$ Then $X_{-}$ is the unique solution
such that $X_{-}+s\overline{A}^{*}\overline{X}_{-}^{-1}A^{*}$ is
invertible for all $s\in \{z:\left \vert z\right \vert >1\}$, or
equivalently,
\begin{equation}
\mathrm{co}\rho \left(  \overline{X}_{-}^{-1}A^{*}\right) \geq1.
\label{eq10}%
\end{equation}

\end{enumerate}
\end{theorem}

\begin{proof}
\textit{Proof of Item 1}: By Theorem \ref{th1},
$W_{+}=X_{+}^{\heartsuit}$ is the maximal solution of the nonlinear
matrix equation (\ref{eq2}). Then according to Lemma \ref{lm8} in
appendix, $W_{+}$ is the unique positive definite solution of
equation (\ref{eq2}) such that $W_{+}+\lambda A^{\lozenge}$ is
nonsingular for all $\lambda \in \{z:\left \vert z\right \vert
<1\}.$ Since $P_{n}^{*}E_{n}P_{n}=E_{n},$ we can compute%
\begin{align*}
W_{+}+\lambda A^{\lozenge}  &  =X_{+}^{\heartsuit}+\lambda
E_{n}A^{\heartsuit
}\\
&  =P_{n}\left[
\begin{array}
[c]{cc}%
X_{+} & 0\\
0 & \overline{X}_{+}%
\end{array}
\right]  P_{n}^{*}+\lambda E_{n}P_{n}\left[
\begin{array}
[c]{cc}%
A & 0\\
0 & \overline{A}%
\end{array}
\right]  P_{n}^{*}\\
&  =P_{n}\left(  \left[
\begin{array}
[c]{cc}%
X_{+} & 0\\
0 & \overline{X}_{+}%
\end{array}
\right]  +\lambda P_{n}^{*}E_{n}P_{n}\left[
\begin{array}
[c]{cc}%
A & 0\\
0 & \overline{A}%
\end{array}
\right]  \right)  P_{n}^{*}\\
&  =P_{n}\left(  \left[
\begin{array}
[c]{cc}%
X_{+} & 0\\
0 & \overline{X}_{+}%
\end{array}
\right]  +\lambda E_{n}\left[
\begin{array}
[c]{cc}%
A & 0\\
0 & \overline{A}%
\end{array}
\right]  \right)  P_{n}^{*}\\
&  =P_{n}\left[
\begin{array}
[c]{cc}%
X_{+} & \lambda \overline{A}\\
\lambda A & \overline{X}_{+}%
\end{array}
\right]  P_{n}^{*},
\end{align*}
from which it follows that%
\begin{equation}
\det \left(  W_{+}+\lambda A^{\lozenge}\right)  =\det \left(
X_{+}\right) \det \left(
X_{+}-\lambda^{2}\overline{A}\overline{X}_{+}^{-1}A\right)  .
\label{eq21}%
\end{equation}
Hence, $\det \left(  W_{+}+\lambda A^{\lozenge}\right)  $ is nonzero
for all $\lambda \in \{z:\left \vert z\right \vert <1\}$ if and only
if $\det (X_{+}-\lambda^{2}\overline{A}\overline{X}_{+}^{-1}A)$ is
nonzero for all
$\lambda \in \{z:\left \vert z\right \vert <1\}.$ Notice that%
\[
\forall \lambda \in \{z:\left \vert z\right \vert <1\}
\Leftrightarrow \forall s\triangleq-\lambda^{2}\in \left \{  z:\left
\vert z\right \vert <1\right \}  .
\]
The first conclusion then follows directly. Moreover, it follows
from
(\ref{eq21}) that%
\begin{align*}
\det \left(  W_{+}+\lambda A^{\lozenge}\right)   &  =\det \left(
X_{+}\right) \det \left(
I_{n}+s\overline{A}\overline{X}_{+}^{-1}AX_{+}^{-1}\right)
\det \left(  X_{+}\right) \\
&  =\left(  \det \left(  X_{+}\right)  \right)  ^{2}\det \left(  I_{n}%
+s\overline{A}\overline{X}_{+}^{-1}AX_{+}^{-1}\right)  ,
\end{align*}
where $s\in \left \{  z:\left \vert z\right \vert <1\right \}  .$
Hence $\det \left(  W_{+}+\lambda A^{\lozenge}\right)  $ is nonzero
for all $\lambda \in \{z:\left \vert z\right \vert <1\}$ if and only
if the matrix $\overline{A}\overline{X}_{+}^{-1}AX_{+}^{-1}$ has no
poles $\alpha
\in \{z:\left \vert z\right \vert >1\}$ which is equivalent to%
\[
1\geq \rho \left(
\overline{A}\overline{X}_{+}^{-1}AX_{+}^{-1}\right) =\rho \left(
\left(  \overline{X}_{+}^{-1}A\right)  \left(  \overline
{\overline{X}_{+}^{-1}A}\right)  \right)  =\rho \left(  \left(
\overline {\overline{X}_{+}^{-1}A}\right)  \left(
\overline{X}_{+}^{-1}A\right) \right)  .
\]
By applying Lemma \ref{lm2}, the above inequality is equivalent to
(\ref{eq9}). Assume that there exists another positive definite
solution $X_{\ast}$ such that $\mathrm{co}\rho \left(
\overline{X}_{+}^{-1}A\right)
\leq1$. As the above process is invertible, we can show that $W_{\ast}%
=X_{\ast}^{\heartsuit}$ is such that $W_{+}+\lambda A^{\lozenge}$ is
nonsingular for all $\lambda \in \{z:\left \vert z\right \vert <1\}$
which contradicts with Lemma \ref{lm8} in appendix.

\textit{Proof of Item 2}: By Theorem \ref{th1},
$W_{-}=X_{-}^{\heartsuit}$ is the minimal solution of the nonlinear
matrix equation (\ref{eq2}). Then according to Lemma \ref{lm8} in
appendix, $W_{-}$ is the unique positive definite solution of
equation (\ref{eq2}) such that $W_{-}+\lambda \left(
A^{\lozenge}\right)  ^{\mathrm{T}}$ is nonsingular for all $\lambda
\in \{z:\left \vert z\right \vert >1\}.$ Similar to the proof of
item 1 of this
theorem, via some computation, we have%
\[
W_{-}+\lambda \left(  A^{\lozenge}\right)  ^{\mathrm{T}}=P_{n}\left[
\begin{array}
[c]{cc}%
X_{-} & \lambda \overline{A}^{*}\\
\lambda A^{*} & \overline{X}_{-}%
\end{array}
\right]  P_{n}^{*}.
\]
The remaining is similar to the proof of item 1 and is omitted for
brevity. The proof is complete.
\end{proof}

\section{\label{sec4}Bound of the Positive Definite Solutions}

In this section, we study the bound of the positive definite
solutions of the nonlinear matrix equation (\ref{eq1}). Let $\left
\{  H_{k}\right \}
_{k=1}^{\infty}$ be generated recursively as follows:%
\begin{equation}
H_{1}=I_{n},\quad H_{k+1}=\left \{
\begin{array}
[c]{l}%
\left[
\begin{BMAT}(@){c1c}{c1c} H_{k} & \left[ \begin{array} [c]{c}\mathbf{0}\\ A^{*}\end{array} \right] \\ \left[ \begin{array} [c]{cc}\mathbf{0} & A \end{array} \right] & I_{n}\end{BMAT}\right]
,\quad k\text{ is odd,}\\
\\
\left[
\begin{BMAT}(@){c1c}{c1c} H_{k} & \left[ \begin{array} [c]{c}\mathbf{0}\\ \overline{A}^{*}\end{array} \right] \\ \left[ \begin{array} [c]{cc}\mathbf{0} & \overline{A}\end{array} \right] & I_{n}\end{BMAT}\right]
,\quad k\text{ is even.}%
\end{array}
\right.  \label{eqH}%
\end{equation}
Here, if $k=2,$ \ the zero matrices in $H_{2}$ are obviously of zero
dimension. Then we can prove the following result.

\begin{theorem}
\label{th7}If the nonlinear matrix equation (\ref{eq1}) has a
positive definite solution, then $H_{k}>0,\forall k\geq1$, and for
any integer
$k\geq1,$ there holds%
\begin{equation}
X>\left[
\begin{array}
[c]{c}%
\mathbf{0}\\
\overline{A}^{*}%
\end{array}
\right]  ^{*}H_{k}^{-1}\left[
\begin{array}
[c]{c}%
\mathbf{0}\\
\overline{A}^{*}%
\end{array}
\right]  \triangleq S_{k},\quad k\text{ is even,} \label{eq31}%
\end{equation}
and%
\begin{equation}
X>\left[
\begin{array}
[c]{c}%
\mathbf{0}\\
\overline{A}^{*}%
\end{array}
\right]  ^{*}\overline{H}_{k}^{-1}\left[
\begin{array}
[c]{c}%
\mathbf{0}\\
\overline{A}^{*}%
\end{array}
\right]  \triangleq S_{k},\quad k\text{ is odd,} \label{eq32}%
\end{equation}
in which, if $k=1,$ the zero matrices involved are of zero
dimensions. Moreover, $S_{k}$ is nondecreasing, namely,
\begin{equation}
S_{k+1}\geq S_{k},\quad \forall k\geq1. \label{eq33}%
\end{equation}

\end{theorem}

\begin{proof}
Since $X=I_{n}-A^{*}\overline{X}^{-1}A>0,$ via a Schur implement, we
have%
\[
0<\left[
\begin{array}
[c]{cc}%
I_{n} & A^{*}\\
A & \overline{X}%
\end{array}
\right]  =\left[
\begin{array}
[c]{cc}%
I_{n} & A^{*}\\
A & I_{n}-\overline{A}^{*}X^{-1}\overline{A}%
\end{array}
\right]  =H_{2}-\left[
\begin{array}
[c]{c}%
0\\
\overline{A}^{*}%
\end{array}
\right]  X^{-1}\left[
\begin{array}
[c]{c}%
0\\
\overline{A}^{*}%
\end{array}
\right]  ^{*},
\]
which indicates that $H_{2}>0.$ Applying another Schur complement on
the above
inequality gives%
\begin{align*}
0<\left[  \begin{BMAT}(@){c1c}{c1c} H_{2} & \left[
\begin{array}
[c]{c}%
0\\
\overline{A}^{*}%
\end{array}
\right]  \\
\left[
\begin{array}
[c]{cc}%
0 & \overline{A}%
\end{array}
\right]   & X
\end{BMAT}\right]  &=\left[  \begin{BMAT}(@){c1c}{c1c}
H_{2} & \left[
\begin{array}
[c]{c}%
0\\
\overline{A}^{*}%
\end{array}
\right]  \\
\left[
\begin{array}
[c]{cc}%
0 & \overline{A}%
\end{array}
\right]   & I_{n}-A^{*}\overline{X}^{-1}A
\end{BMAT}\right] \\ &=H_{3}-\left[
\begin{array}
[c]{c}%
\mathbf{0}\\
A^{*}%
\end{array}
\right]  \overline{X}^{-1}\left[
\begin{array}
[c]{c}%
\mathbf{0}\\
A^{*}%
\end{array}
\right]  ^{*},
\end{align*}
which indicates that $H_{3}>0.$ By using Schur complement again, the
above
inequality can be equivalently rewritten as%
\begin{align*}
0<\left[  \begin{BMAT}(@){c1c}{c1c} H_{3} & \left[
\begin{array}
[c]{c}%
0\\
A^{*}%
\end{array}
\right]  \\
\left[
\begin{array}
[c]{cc}%
\mathbf{0} & A
\end{array}
\right]   & \overline{X}%
\end{BMAT}\right]  &=\left[  \begin{BMAT}(@){c1c}{c1c}
H_{3} & \left[
\begin{array}
[c]{c}%
\mathbf{0}\\
A^{*}%
\end{array}
\right]  \\
\left[
\begin{array}
[c]{cc}%
\mathbf{0} & A
\end{array}
\right]   & I_{n}-\overline{A}^{*}X^{-1}\overline{A}%
\end{BMAT}\right] \\ & =H_{4}-\left[
\begin{array}
[c]{c}%
\mathbf{0}\\
\overline{A}^{*}%
\end{array}
\right]  X^{-1}\left[
\begin{array}
[c]{c}%
\mathbf{0}\\
\overline{A}^{*}%
\end{array}
\right]  ^{*},
\end{align*}
which indicates that $H_{4}>0.$ Repeating the above process gives
$H_{k}>0,\forall k\geq1.$

On the other hand, from the above development, we see that, if $k$
is even, then
\[
\left[  \begin{BMAT}(@){c1c}{c1c} H_{k} & \left[
\begin{array}
[c]{c}%
0\\
\overline{A}^{*}%
\end{array}
\right]  \\
\left[
\begin{array}
[c]{cc}%
0 & \overline{A}%
\end{array}
\right]   & X
\end{BMAT}\right]  >0,
\]
and if $k$ if odd, then%
\[
\left[  \begin{BMAT}(@){c1c}{c1c} H_{k} & \left[
\begin{array}
[c]{c}%
\mathbf{0}\\
A^{*}%
\end{array}
\right]  \\
\left[
\begin{array}
[c]{cc}%
\mathbf{0} & A
\end{array}
\right]   & \overline{X}%
\end{BMAT}\right]  >0,
\]
where, if $k=1,$ the zero matrices are of zero dimension. We notice
that the above two inequalities are respectively equivalent to
(\ref{eq31}) and (\ref{eq32}) via Schur complements.

We finally show (\ref{eq33}). We only prove the case that $k$ is odd
since the
case that $k$ is even can be proven similarly. As $k+1$ is even, we have%
\begin{align}
S_{k+1}  &  =\left[
\begin{array}
[c]{c}%
\mathbf{0}\\
\overline{A}^{*}%
\end{array}
\right]  ^{*}H_{k+1}^{-1}\left[
\begin{array}
[c]{c}%
\mathbf{0}\\
\overline{A}^{*}%
\end{array}
\right] \nonumber \\
&  =\left[
\begin{array}
[c]{c}%
\mathbf{0}\\
\overline{A}^{*}%
\end{array}
\right]  ^{*}\left[
\begin{BMAT}(@){c1c}{c1c} H_{k} & \left[ \begin{array} [c]{c}\mathbf{0}\\ A^{*}\end{array} \right] \\ \left[ \begin{array} [c]{cc}\mathbf{0} & A \end{array} \right] & I_{n}\end{BMAT}\right]
^{-1}\left[
\begin{array}
[c]{c}%
\mathbf{0}\\
\overline{A}^{*}%
\end{array}
\right]  . \label{eq36}%
\end{align}
We notice that, if $k$ is odd, then%
\begin{align*}
\left[
\begin{BMAT}(@){c1c}{c1c} H_{k} & \left[ \begin{array} [c]{c}\mathbf{0}\\ A^{*}\end{array} \right] \\ \left[ \begin{array} [c]{cc}\mathbf{0} & A \end{array} \right] & I_{n}\end{BMAT}\right]
&  =\left[
\begin{BMAT}(@){c1c}{c1c} I_{n} & \left[ \begin{array} [c]{cc}A^{*} & \mathbf{0}\end{array} \right] \\ \left[ \begin{array} [c]{c}A\\ \mathbf{0}\end{array} \right] & \overline{H}_{k}\end{BMAT}\right]
\\
&  =\left[
\begin{array}
[c]{cc}%
\triangledown_{11} & \triangledown_{12}\\
\triangledown_{12}^{*} & \left(  \overline{H}_{k}-\left[
\begin{array}
[c]{c}%
A\\
\mathbf{0}%
\end{array}
\right]  \left[
\begin{array}
[c]{cc}%
A^{*} & \mathbf{0}%
\end{array}
\right]  \right)  ^{-1}%
\end{array}
\right]  ^{-1},
\end{align*}
where $\triangledown_{11}$ and $\triangledown_{12}$ are two matrices
that are
of no concern. With this, the relation in (\ref{eq36}) can be continued as%
\begin{align*}
S_{k+1}  &  =\left[
\begin{BMAT}(c){c}{c1c} \mathbf{0}\\ \left[ \begin{array} [c]{c}\mathbf{0}\\ \overline{A}^{*}\end{array} \right] \end{BMAT}\right]
^{*}\left[
\begin{BMAT}(@){c1c}{c1c} I_{n} & \left[ \begin{array} [c]{cc}A^{*} & \mathbf{0}\end{array} \right] \\ \left[ \begin{array} [c]{c}A\\ \mathbf{0}\end{array} \right] & \overline{H}_{k}\end{BMAT}\right]
^{-1}\left[
\begin{BMAT}(c){c}{c1c} \mathbf{0}\\ \left[ \begin{array} [c]{c}\mathbf{0}\\ \overline{A}^{*}\end{array} \right] \end{BMAT}\right]
\\
&  =\left[
\begin{array}
[c]{c}%
\mathbf{0}\\
\overline{A}^{*}%
\end{array}
\right]  ^{*}\left(  \overline{H}_{k}-\left[
\begin{array}
[c]{c}%
A\\
\mathbf{0}%
\end{array}
\right]  \left[
\begin{array}
[c]{cc}%
A^{*} & \mathbf{0}%
\end{array}
\right]  \right)  ^{-1}\left[
\begin{array}
[c]{c}%
\mathbf{0}\\
\overline{A}^{*}%
\end{array}
\right] \\
&  \geq \left[
\begin{array}
[c]{c}%
\mathbf{0}\\
\overline{A}^{*}%
\end{array}
\right]  ^{*}\overline{H}_{k}^{-1}\left[
\begin{array}
[c]{c}%
\mathbf{0}\\
\overline{A}^{*}%
\end{array}
\right] \\
&  =S_{k}.
\end{align*}
In the above development, the zero matrices involved are of zero
dimension if $k=1.$ The proof is finished.
\end{proof}

Since $S_{k}$ is bounded above, there must exist a positive definite
matrix
$S_{\infty}<I_{n}$ such that%
\[
\lim_{k\rightarrow \infty}S_{k}=S_{\infty}.
\]
Obviously, we have $X>S_{\infty}.$ It is not clear whether
$S_{\infty}=X_{-}.$ Moreover, it is clear that the larger the $k$,
the better the $S_{k}$ gives a lower bound of $X.$ However, large
$k$ may lead to numerical problems. By choosing some special values
in $k$, the following corollary can be obtained.

\begin{corollary}
If the nonlinear matrix equation (\ref{eq1}) has a positive definite
solution,
then $A$ satisfies%
\begin{equation}
I_{n}>AA^{*}+\overline{A^{*}A}. \label{eq38}%
\end{equation}
Moreover, the solution $X$ satisfies the following inequalities%
\begin{align}
X  &  >S_{1}=\overline{AA^{*}},\label{eq24}\\
X  &  >S_{2}=\overline{A}\left(  I_{n}-AA^{*}\right)  ^{-1}%
A^{\mathrm{T}},\label{eq26}\\
X  &  >S_{3}=\overline{A}\left(  I_{n}-A\left(  I_{n}-\overline{AA^{*%
}}\right)  ^{-1}A^{*}\right)  ^{-1}A^{\mathrm{T}}. \label{eq27}%
\end{align}

\end{corollary}

\begin{proof}
Let $k=2.$ We get from Theorem \ref{th7} that $H_{2}>0,$ which, via
a Schur complement, is equivalent to $I_{n}>AA^{*}$. Similarly, by
letting
$k=3,$ we get from $H_{3}>0$ that%
\[
0<\left[
\begin{array}
[c]{ccc}%
I_{n} & A^{*} & \mathbf{0}\\
A & I_{n} & \overline{A}^{*}\\
\mathbf{0} & \overline{A} & I_{n}%
\end{array}
\right]  ,
\]
which, via a Schur complement, implies%
\[
I_{n}-\left[
\begin{array}
[c]{cc}%
0 & \overline{A}%
\end{array}
\right]  \left[
\begin{array}
[c]{cc}%
I_{n} & A^{*}\\
A & I_{n}%
\end{array}
\right]  ^{-1}\left[
\begin{array}
[c]{c}%
0\\
\overline{A}^{*}%
\end{array}
\right]  =I_{n}-\overline{A}\left(  I_{n}-AA^{*}\right)
^{-1}A^{\mathrm{T}}>0.
\]
By applying Schur complement again, the above inequality is equivalent to%
\[
\left[
\begin{array}
[c]{cc}%
I_{n} & \overline{A}\\
A^{\mathrm{T}} & I_{n}-AA^{*}%
\end{array}
\right]  >0.
\]
However, another Schur complement indicates that the above
inequality implies $A^{*}A+\overline{AA^{*}}<I_{n}$.

Now if we set $k=1,$ we get from (\ref{eq32}) that $\overline{X}%
>AA^{*}=\overline{S}_{1}$ (or $X>\overline{A}A^{\mathrm{T}}=S_{1}$)
which is (\ref{eq24}); if we set $k=2,$ we get from (\ref{eq31}) that%
\begin{align*}
X  &  >S_{2}=\left[
\begin{array}
[c]{cc}%
0 & \overline{A}%
\end{array}
\right]  \left[
\begin{array}
[c]{cc}%
I_{n} & A^{*}\\
A & I_{n}%
\end{array}
\right]  ^{-1}\left[
\begin{array}
[c]{c}%
0\\
\overline{A}^{*}%
\end{array}
\right] \\
&  =\left[
\begin{array}
[c]{cc}%
0 & \overline{A}%
\end{array}
\right]  \left[
\begin{array}
[c]{cc}%
\triangledown_{11} & \triangledown_{12}\\
\triangledown_{12}^{*} & \left(  I_{n}-AA^{*}\right)  ^{-1}%
\end{array}
\right]  \left[
\begin{array}
[c]{c}%
0\\
\overline{A}^{*}%
\end{array}
\right] \\
&  =\overline{A}\left(  I_{n}-AA^{*}\right) ^{-1}A^{\mathrm{T}},
\end{align*}
which is (\ref{eq26}). Here $\triangledown_{ij}$ denotes the
elements that are of no concern. Moreover, if we set $k=3,$ we know
from (\ref{eq32}) that
\begin{align*}
X  &  >S_{3}=\left[
\begin{array}
[c]{c}%
\mathbf{0}\\
\overline{A}^{*}%
\end{array}
\right]  ^{*}\overline{H}_{3}^{-1}\left[
\begin{array}
[c]{c}%
\mathbf{0}\\
\overline{A}^{*}%
\end{array}
\right] \\
&  =\left[
\begin{array}
[c]{cc}%
\mathbf{0} & \overline{A}%
\end{array}
\right]  \left[
\begin{array}
[c]{ccc}%
I_{n} & \overline{A}^{*} & \mathbf{0}\\
\overline{A} & I_{n} & A^{*}\\
\mathbf{0} & A & I_{n}%
\end{array}
\right]  ^{-1}\left[
\begin{array}
[c]{c}%
0\\
\overline{A}^{*}%
\end{array}
\right] \\
&  =\left[
\begin{array}
[c]{cc}%
\mathbf{0} & \overline{A}%
\end{array}
\right]  \left[
\begin{array}
[c]{cc}%
\triangledown_{11} & \triangledown_{12}\\
\triangledown_{12}^{*} & \left(  I_{n}-A\left( I_{n}-\overline
{AA^{*}}\right)  ^{-1}A^{*}\right)  ^{-1}%
\end{array}
\right]  \left[
\begin{array}
[c]{c}%
\mathbf{0}\\
\overline{A}^{*}%
\end{array}
\right] \\
&  =\overline{A}\left(  I_{n}-A\left(  I_{n}-\overline{AA^{*}%
}\right)  ^{-1}A^{*}\right)  ^{-1}A^{\mathrm{T}},
\end{align*}
which is just (\ref{eq27}). The proof is finished.
\end{proof}

With the help of Lemma \ref{lm1} and Theorem \ref{th7}, we can also
obtain upper bounds of $X.$ To this end, we let $\left \{
G_{k}\right \} _{k=1}^{\infty}$ be generated as (\ref{eqH}) where
$A^{*}$ is
replaced with $A,$ namely,%
\[
G_{1}=I_{n},\quad G_{k+1}=\left \{
\begin{array}
[c]{l}%
\left[
\begin{BMAT}(@){c1c}{c1c} G_{k} & \left[ \begin{array} [c]{c}\mathbf{0}\\ A \end{array} \right] \\ \left[ \begin{array} [c]{cc}\mathbf{0} & A^{*}\end{array} \right] & I_{n}\end{BMAT}\right]
,\quad k\text{ is odd}\\
\\
\left[
\begin{BMAT}(@){c1c}{c1c} G_{k} & \left[ \begin{array} [c]{c}\mathbf{0}\\ \overline{A}\end{array} \right] \\ \left[ \begin{array} [c]{cc}\mathbf{0} & \overline{A}^{*}\end{array} \right] & I_{n}\end{BMAT}\right]
,\quad k\text{ is even.}%
\end{array}
\right.
\]
Again, if $k=2,$\ the zero matrices in $G_{2}$ are obviously of zero
dimension.

\begin{theorem}
Assume that $A$ is nonsingular and the nonlinear matrix equation in
(\ref{eq1}) has a positive definite solution. Then $G_{k}>0,\forall
k\geq1$,
and for any integer $k\geq1,$ there holds%
\begin{equation}
X<I_{n}-\left[
\begin{array}
[c]{c}%
\mathbf{0}\\
A
\end{array}
\right]  ^{*}\overline{G}_{k}^{-1}\left[
\begin{array}
[c]{c}%
\mathbf{0}\\
A
\end{array}
\right]  \triangleq R_{k},\quad k\text{ is even,} \label{eq37}%
\end{equation}
and%
\begin{equation}
X<I_{n}-\left[
\begin{array}
[c]{c}%
\mathbf{0}\\
A
\end{array}
\right]  ^{*}G_{k}^{-1}\left[
\begin{array}
[c]{c}%
\mathbf{0}\\
A
\end{array}
\right]  \triangleq R_{k},\quad k\text{ is odd.} \label{eq34}%
\end{equation}
in which, if $k=1,$ the zero matrices involved are of zero
dimension. Moreover, $R_{k}$ is non-increasing, namely,
\begin{equation}
R_{k+1}\leq R_{k},\quad \forall k\geq1. \label{eq35}%
\end{equation}

\end{theorem}

\begin{proof}
Since $A$ is nonsingular, by Lemma \ref{lm1}, equation (\ref{eq1})
is equivalent to equation (\ref{eq5}) with $Y=I-\overline{X}.$
Applying Theorem
\ref{th7} on equation (\ref{eq5}) gives $G_{k}>0,k\geq1$ and%
\begin{align*}
Y  &  >\left[
\begin{array}
[c]{c}%
\mathbf{0}\\
\overline{A}%
\end{array}
\right]  ^{*}G_{k}^{-1}\left[
\begin{array}
[c]{c}%
\mathbf{0}\\
\overline{A}%
\end{array}
\right]  \triangleq S_{k}^{\prime},\quad k\text{ is even,}\\
Y  &  >\left[
\begin{array}
[c]{c}%
\mathbf{0}\\
\overline{A}%
\end{array}
\right]  ^{*}\overline{G}_{k}^{-1}\left[
\begin{array}
[c]{c}%
\mathbf{0}\\
\overline{A}%
\end{array}
\right]  \triangleq S_{k}^{\prime},\quad k\text{ is odd,}%
\end{align*}
which are respectively equivalent to (\ref{eq37}) and (\ref{eq34})
by using $Y=I-\overline{X}.$ Finally, $R_{k}=I_{n}-S_{k}^{\prime}$
is non-increasing because $S_{k}^{\prime}$ is nondecreasing
according to Theorem \ref{th7}. The proof is finished.
\end{proof}

It is also clear that the larger the $k$, the better the $R_{k}$
gives an upper bound of $X.$ However, large $k$ may lead to
numerical problems. Choosing some special values in $k$ gives the
following corollary.

\begin{corollary}
If the nonlinear matrix equation in (\ref{eq1}) has a positive
definite
solution and $A$ is nonsingular, then $A$ satisfies%
\[
I_{n}>A^{*}A+\overline{AA^{*}}.
\]
Moreover, the solution $X$ satisfies the following inequalities%
\begin{align*}
X  &  <R_{1}=I_{n}-A^{*}A,\\
X  &  <R_{2}=I_{n}-A^{*}\left(  I_{n}-\overline{A^{*}%
A}\right)  ^{-1}A,\\
X  &  <R_{3}=I_{n}-A^{*}\left(  I_{n}-A^{\mathrm{T}}\left(
I_{n}-A^{*}A\right)  ^{-1}\overline{A}\right)  ^{-1}A.
\end{align*}

\end{corollary}

\section{\label{sec5}Sufficient Conditions and Necessary Conditions}

In this section, we present some necessary conditions and sufficient
conditions for the existence of a positive definite solutions of the
nonlinear matrix equation (\ref{eq1}).

\begin{theorem}
If the nonlinear matrix equation (\ref{eq1}) has a positive definite
solution, then $\rho \left(  A\overline{A}\right)  \leq
\frac{1}{4},\left \Vert
A\right \Vert <1$ and%
\begin{equation}
\rho \left(  \left(  A\pm \overline{A}^{*}\right)  \left(
\overline{A\pm \overline{A}^{*}}\right)  \right)  \leq1, \label{eq30}%
\end{equation}
which can be equivalently rewritten as $\mathrm{co}\rho \left(  A\pm
\overline{A}^{*}\right)  \leq1.$
\end{theorem}

\begin{proof}
It follows from Theorem \ref{th1} and Lemma \ref{lm12} in appendix
that if the nonlinear matrix equation (\ref{eq1}) has a positive
definite solution, then $\rho \left(  A^{\lozenge}\right)  \leq
\frac{1}{2},\left \Vert A^{\lozenge }\right \Vert <1$ and
\begin{equation}
\rho \left(  A^{\lozenge}\pm \left(  A^{\lozenge}\right)
^{\mathrm{T}}\right)
\leq1. \label{eq28}%
\end{equation}
Clearly, by applying Lemma \ref{lm0}, $\rho \left(
A^{\lozenge}\right) \leq \frac{1}{2}$ is equivalent to $\rho \left(
A\overline{A}\right)  \leq \frac{1}{4}$ and $\left \Vert
A^{\lozenge}\right \Vert <1$ is equivalent to $\left \Vert A\right
\Vert <1$ (we point out that $\left \Vert A\right \Vert <1$ also
follows directly from (\ref{eq38}))$.$ We next show that
(\ref{eq28}) is equivalent to (\ref{eq30}). By virtue of Lemma
\ref{lm0} and in view of $P_{n}^{*}E_{n}P_{n}=E_{n}$, we get
\begin{align*}
\rho \left(  A^{\lozenge}\pm \left(  A^{\lozenge}\right)
^{\mathrm{T}}\right) &  =\rho \left(  E_{n}A^{\heartsuit}\pm \left(
A^{\heartsuit}\right)
^{\mathrm{T}}E_{n}\right) \\
&  =\rho \left(  E_{n}P_{n}\left[
\begin{array}
[c]{cc}%
A & 0\\
0 & \overline{A}%
\end{array}
\right]  P_{n}^{*}\pm P_{n}\left[
\begin{array}
[c]{cc}%
A^{*} & 0\\
0 & \overline{A}^{*}%
\end{array}
\right]  P_{n}^{*}E_{n}\right) \\
&  =\rho \left(  P_{n}^{*}E_{n}P_{n}\left[
\begin{array}
[c]{cc}%
A & 0\\
0 & \overline{A}%
\end{array}
\right]  \pm \left[
\begin{array}
[c]{cc}%
A^{*} & 0\\
0 & \overline{A}^{*}%
\end{array}
\right]  P_{n}^{*}E_{n}P_{n}\right) \\
&  =\rho \left(  E_{n}\left[
\begin{array}
[c]{cc}%
A & 0\\
0 & \overline{A}%
\end{array}
\right]  \pm \left[
\begin{array}
[c]{cc}%
A^{*} & 0\\
0 & \overline{A}^{*}%
\end{array}
\right]  E_{n}\right) \\
&  =\rho \left(  \left[
\begin{array}
[c]{cc}%
0 & \overline{A}\pm A^{*}\\
A\pm \overline{A}^{*} & 0
\end{array}
\right]  \right) \\
&  =\rho^{\frac{1}{2}}\left(  \left[
\begin{array}
[c]{cc}%
0 & \overline{A}\pm A^{*}\\
A\pm \overline{A}^{*} & 0
\end{array}
\right]  ^{2}\right) \\
&  =\rho^{\frac{1}{2}}\left(  \left[
\begin{array}
[c]{cc}%
\left(  \overline{A}\pm A^{*}\right)  \left(  A\pm \overline
{A}^{*}\right)  & 0\\
0 & \left(  A\pm \overline{A}^{*}\right)  \left( \overline{A}\pm
A^{*}\right)
\end{array}
\right]  \right) \\
&  =\rho^{\frac{1}{2}}\left(  \left(  A\pm \overline{A}^{*}\right)
\left(  \overline{A}\pm A^{*}\right)  \right)  ,
\end{align*}
which is\ the desired relation. Finally, the equivalence between
(\ref{eq30}) and $\mathrm{co}\rho \left(  A\pm
\overline{A}^{*}\right)  \leq1$ follows from Lemma \ref{lm2}. The
proof is finished.
\end{proof}

\begin{theorem}
The nonlinear matrix equation (\ref{eq1}) has a positive definite
solution provided
\begin{equation}
\left \Vert A\right \Vert \leq \frac{1}{2}. \label{eq39}%
\end{equation}
Moreover, if $A$ is con-normal, then the nonlinear matrix equation
(\ref{eq1}) has a positive definite solution if and only if $A$
satisfies (\ref{eq39}). In this case, the maximal solution is given
by
\begin{equation}
X_{+}=\frac{1}{2}\left(  I_{n}+\left(  I_{n}-4A^{*}A\right)
^{\frac{1}{2}}\right)  . \label{eq40}%
\end{equation}
If $A$ is further assumed to be nonsingular, then the minimal
solution can be
expressed as%
\begin{equation}
X_{-}=\frac{1}{2}\left(  I_{n}-\left(  I_{n}-4A^{*}A\right)
^{\frac{1}{2}}\right)  . \label{eq41}%
\end{equation}

\end{theorem}

\begin{proof}
If $\left \Vert A\right \Vert \leq \frac{1}{2},$ then by Lemma
\ref{lm0}, we know that $\left \Vert A^{\lozenge}\right \Vert \leq
\frac{1}{2},$ which,\ by using Lemma \ref{lm14}, indicates that
equation (\ref{eq2}) has positive definite solution. This is further
equivalent to the existence of positive definite solution of
equation (\ref{eq1}) in view of Theorem \ref{th1}. The case that $A$
is con-normal can be shown similarly. We next show (\ref{eq40}).
Notice that, according to Lemma \ref{lm14}, the maximal solution of
equation
(\ref{eq2}) is given as%
\begin{align*}
W_{+}  &  =\frac{1}{2}\left(  I_{2n}+\left(  I_{2n}-4\left(
A^{\lozenge
}\right)  ^{\mathrm{T}}A^{\lozenge}\right)  ^{\frac{1}{2}}\right) \\
&  =\frac{1}{2}\left(  I_{2n}+\left(  I_{2n}-4\left(
A^{\heartsuit}\right)
^{\mathrm{T}}E_{n}^{\mathrm{T}}E_{n}A^{\heartsuit}\right)  ^{\frac{1}{2}%
}\right) \\
&  =\frac{1}{2}\left(  I_{2n}+\left(  I_{2n}-4\left(
A^{\heartsuit}\right)
^{\mathrm{T}}A^{\heartsuit}\right)  ^{\frac{1}{2}}\right) \\
&  =\frac{1}{2}\left(  I_{2n}+\left(  I_{2n}-4\left( A^{*}A\right)
^{\heartsuit}\right)  ^{\frac{1}{2}}\right) \\
&  =\frac{1}{2}\left(  I_{n}+\left(  I_{n}-4A^{*}A\right) ^{\frac
{1}{2}}\right)  ^{\heartsuit},
\end{align*}
which, according to Theorem \ref{th1}, implies (\ref{eq40}). The
equation (\ref{eq41}) can be shown similarly. The proof is done.
\end{proof}

\section{\label{sec6}Conclusion}

This paper has studied the existence of a positive definite
solutions to the nonlinear matrix equation
$X+A^{*}\overline{X}^{-1}A=Q$. With the help of some operators
associated with complex matrices, we have shown that the existence
of a positive definite solution of this type of nonlinear matrix
equation is equivalent to the existence of a positive definite
solution of a nonlinear matrix equation in the form of
$W+B^{\mathrm{T}}W^{-1}B=I$ where $B$ is real and is determined by
$A$. Since the later nonlinear matrix equation has been well studied
in the literature, properties of the original nonlinear matrix
equations can be established based on the existing results on the
transformed nonlinear matrix equations. Moreover, with the help of
Schur complement, we have shown in this paper some upper bounds and
lower bounds on the solutions to the nonlinear matrix equations.
Simultaneously, some easily tested sufficient conditions and
necessary conditions for the existence of positive definite solution
of the nonlinear equations have also been established. We point out
that, by combining the results obtained in this paper and the
existing results on numerical computation of solutions to the
standard nonlinear matrix equation $W+B^{\mathrm{T}}W^{-1}B=I$,
numerical reliable algorithms can be built for computing the
positive definite solutions to the original nonlinear matrix
equation, which is currently under study.

\section*{Appendix}

\subsection*{A: Solutions of Matrix Equation $X+A^{*}X^{-1}A=Q$}

In this subsection, we recall some basic results regarding positive
definite
solutions of the following matrix equation%
\begin{equation}
X+A^{*}X^{-1}A=Q. \label{EqNonlinear}%
\end{equation}

\begin{lemma}
\label{lm8}(Theorem 3.4 in \cite{err93laa}) Suppose that $Q>0$ and
assume that he nonlinear matrix equation (\ref{EqNonlinear}) has a
positive definite solution. Then it has a maximal and minimal
solution $X_{+}$ and $X_{-},$ respectively. Moreover, $X_{+}$ is the
unique solution for which $X+\lambda A$ is invertible for all $\left
\vert \lambda \right \vert <1$, while $X_{-}$ is the unique solution
for which $X+\lambda A^{*}$ is invertible for all $\left \vert
\lambda \right \vert >1.$
\end{lemma}

\begin{lemma}
\label{lm9}(Algorithm 4.1 in \cite{err93laa}) Suppose that $Q=I.$ If
the nonlinear matrix equation (\ref{EqNonlinear}) has a positive
definite
solution, then the iteration%
\[
X_{k+1}=I_{n}-A^{*}X_{k}^{-1}A,\quad X_{0}=I_{n},
\]
converges to the maximal solution $X_{+},$ namely,
$\lim_{k\rightarrow \infty }X_{k}=X_{+}.$
\end{lemma}

\begin{lemma}
\label{lm11}(Theorem 8.1 in \cite{err93laa}) Suppose that $Q=I$ and
$A$ is real. If the nonlinear matrix equation (\ref{EqNonlinear})
has a positive definite solution, then $X_{+}$ is real. Furthermore,
if $A$ is nonsingular, then $X_{-}$ is also real.
\end{lemma}

\begin{lemma}
\label{lm10}(Theorem 5.1 in \cite{err93laa}) Suppose that $A$ is
invertible. Then the nonlinear matrix equation (\ref{EqNonlinear})
has a positive definite solution if and only if $\omega \left(
A\right)  \leq \frac{1}{2}.$
\end{lemma}

\begin{lemma}
\label{lm12}(Theorem 7 in \cite{e93laa} and Theorem 3.1 in
\cite{zx96laa}) If the nonlinear matrix equation (\ref{EqNonlinear})
has a positive definite solution, then $\rho \left(  A\right)  \leq
\frac{1}{2},\left \Vert A\right \Vert <1$ and $\rho \left(  A\pm
A^{*}\right)  \leq1.$
\end{lemma}

\begin{lemma}
\label{lm14}(Theorem 11 and Theorem 13 in \cite{e93laa}) The
nonlinear matrix equation (\ref{EqNonlinear}) has a positive
definite solution provided $\left \Vert A\right \Vert \leq
\frac{1}{2}.$ Moreover, if $A$ is normal, then the nonlinear matrix
equation (\ref{EqNonlinear}) has a positive definite solution if and
only if $\left \Vert A\right \Vert \leq \frac{1}{2}.$ In this
case, the maximal solution is given by%
\[
X_{+}=\frac{1}{2}\left(  I_{n}+\left(  I_{n}-4A^{\mathrm{T}}A\right)
^{\frac{1}{2}}\right)  .
\]
Furthermore, if $A$ is nonsingular, then%
\[
X_{-}=\frac{1}{2}\left(  I_{n}-\left(  I_{n}-4A^{\mathrm{T}}A\right)
^{\frac{1}{2}}\right)  .
\]

\end{lemma}

\subsection*{B: Proof of Lemma \ref{lm0}}

\textit{Proof of Item 1:} These equalities can be verified directly
by definition.

\textit{Proof of Item 2}. This result follows from Lemma 9 in
\cite{zld11laa}.

\textit{Proof of Item 3:} We need only to show that $A$ is a normal
matrix if and only if $A^{\heartsuit}$ is a real normal matrix since
unitary matrix is a special case of normal matrix. If
$A^{\heartsuit}$ is a real normal matrix,
then $\left(  A^{\heartsuit}\right)  ^{\mathrm{T}}A^{\heartsuit}%
=A^{\heartsuit}\left(  A^{\heartsuit}\right)  ^{\mathrm{T}}.$
However, $A^{\heartsuit}\left(  A^{\heartsuit}\right)
^{\mathrm{T}}=\left( AA^{*}\right)  ^{\heartsuit}$ and $\left(
A^{\heartsuit}\right) ^{\mathrm{T}}A^{\heartsuit}=\left(
A^{*}A\right)  ^{\heartsuit}.$ Hence we have $\left( AA^{*}\right)
^{\heartsuit}=\left(
A^{*}A\right)  ^{\heartsuit}$ and consequently $AA^{*%
}=A^{*}A,$ that is, $A$ is a normal matrix. The converse can be
shown similarly.

\textit{Proof of Item 4:} From item 2 of this lemma, we obtain%
\begin{align*}
\rho \left(  A^{\heartsuit}\right)   &  =\rho \left(  P_{n}\left[
\begin{array}
[c]{cc}%
A & 0\\
0 & \overline{A}%
\end{array}
\right]  P_{n}^{*}\right) \\
&  =\rho \left(  \left[
\begin{array}
[c]{cc}%
A & 0\\
0 & \overline{A}%
\end{array}
\right]  \right) \\
&  =\max \left \{  \rho \left(  A\right)  ,\rho \left(
\overline{A}\right)
\right \} \\
&  =\rho \left(  A\right)  .
\end{align*}
Similarly, we can compute%
\begin{align*}
\rho \left(  A^{\lozenge}\right)   &  =\rho \left(
E_{n}A^{\heartsuit}\right) =\rho \left(  E_{n}P_{n}\left[
\begin{array}
[c]{cc}%
A & 0\\
0 & \overline{A}%
\end{array}
\right]  P_{n}^{*}\right) \\
&  =\rho \left(  P_{n}^{*}E_{n}P_{n}\left[
\begin{array}
[c]{cc}%
A & 0\\
0 & \overline{A}%
\end{array}
\right]  P_{n}^{*}P_{n}\right) \\
&  =\rho \left(  E_{n}\left[
\begin{array}
[c]{cc}%
A & 0\\
0 & \overline{A}%
\end{array}
\right]  \right)  =\rho \left(  \left[
\begin{array}
[c]{cc}%
0 & \overline{A}\\
A & 0
\end{array}
\right]  \right) \\
&  =\rho^{\frac{1}{2}}\left(  \left[
\begin{array}
[c]{cc}%
0 & \overline{A}\\
A & 0
\end{array}
\right]  \left[
\begin{array}
[c]{cc}%
0 & \overline{A}\\
A & 0
\end{array}
\right]  \right) \\
&  =\rho^{\frac{1}{2}}\left(  \left[
\begin{array}
[c]{cc}%
\overline{A}A & 0\\
0 & A\overline{A}%
\end{array}
\right]  \right)  =\rho^{\frac{1}{2}}\left(  \overline{A}A\right)  .
\end{align*}

\textit{Proof of Item 5: }By definition, we can compute%
\begin{align*}
\left(  A^{\lozenge}\right)  ^{\mathrm{T}}A^{\lozenge}&=\left(  E_{n}%
A^{\heartsuit}\right)  ^{\mathrm{T}}E_{n}A^{\heartsuit}=\left(
A^{\heartsuit }\right)
^{\mathrm{T}}E_{n}^{\mathrm{T}}E_{n}A^{\heartsuit}\\
&=\left(
A^{\heartsuit}\right)  ^{\mathrm{T}}A^{\heartsuit}=\left(  A^{*%
}\right)  ^{\heartsuit}A^{\heartsuit}=\left(  A^{*}A\right)
^{\heartsuit},
\end{align*}
and similarly,%
\begin{align*} A^{\lozenge}\left(  A^{\lozenge}\right)
^{\mathrm{T}}&=E_{n}A^{\heartsuit }\left( E_{n}A^{\heartsuit}\right)
^{\mathrm{T}}=E_{n}A^{\heartsuit}\left( A^{\heartsuit}\right)
^{\mathrm{T}}E_{n}^{\mathrm{T}}\\
&=E_{n}A^{\heartsuit }\left(
A^{*}\right)  ^{\heartsuit}E_{n}^{\mathrm{T}}=E_{n}\left(
AA^{*}\right)  ^{\heartsuit}E_{n}^{\mathrm{T}}=\left( \overline
{AA^{*}}\right)  ^{\heartsuit}.
\end{align*}
Clearly, $A^{\lozenge}$ is a normal matrix if and only if $\left(
A^{\lozenge}\right)  ^{\mathrm{T}}A^{\lozenge}=A^{\lozenge}\left(
A^{\lozenge}\right)  ^{\mathrm{T}},$ which is equivalent to $\left(
\overline{AA^{*}}\right)  ^{\heartsuit}=\left(  A^{*%
}A\right)  ^{\heartsuit}$, namely, $A^{*}A=\overline{AA^{*}%
}.$

\textit{Proof of Item 6:} By using item 1 of this lemma, we obtain%
\[
\left \Vert A^{\lozenge}\right \Vert =\left \Vert
E_{n}A^{\heartsuit}\right \Vert =\left \Vert A^{\heartsuit}\right
\Vert =\left \Vert P_{n}\left[
\begin{array}
[c]{cc}%
A & 0\\
0 & \overline{A}%
\end{array}
\right]  P_{m}^{*}\right \Vert =\left \Vert \left[
\begin{array}
[c]{cc}%
A & 0\\
0 & \overline{A}%
\end{array}
\right]  \right \Vert =\left \Vert A\right \Vert .
\]

\textit{Proof of Item 7}:\ Let $U$ be a unitary matrix such that
$A=UDU^{*}$ where $D$ is a real diagonal positive semi-definite
matrix. Then%
\begin{align*}
\left(  A^{\heartsuit}\right)  ^{\frac{1}{2}}  &  =\left(  \left(
UDU^{*}\right)  ^{\heartsuit}\right)  ^{\frac{1}{2}}=\left(
U^{\heartsuit}\left[
\begin{array}
[c]{cc}%
D & 0\\
0 & D
\end{array}
\right]  \left(  U^{\heartsuit}\right)  ^{\mathrm{T}}\right)  ^{\frac{1}{2}}\\
&  =U^{\heartsuit}\left[
\begin{array}
[c]{cc}%
D^{\frac{1}{2}} & 0\\
0 & D^{\frac{1}{2}}%
\end{array}
\right]  \left(  U^{\heartsuit}\right)  ^{\mathrm{T}}=\left(  UD^{\frac{1}{2}%
}U^{*}\right)  ^{\heartsuit}=\left(  A^{\frac{1}{2}}\right)
^{\heartsuit}.
\end{align*}

\subsection*{C: Proof of Lemma \ref{lm2}}

We only show the case \textquotedblleft$\leq$\textquotedblright.
According to the results in \cite{bhrh87ctmt}, we know that

(1). for $\lambda \geq0,\lambda \in \mathrm{co}\lambda \left(
A\right) \Leftrightarrow \lambda^{2}\in \lambda \left(
A\overline{A}\right)  .$

(2). for $\lambda<0,\lambda \in \mathrm{co}\lambda \left(  A\right)
\Leftrightarrow \lambda \in \lambda \left(  A\overline{A}\right)  .$

(3). for $\operatorname{Im}\left(  \lambda \right)  \neq0,\lambda
\in \mathrm{co}\lambda \left(  A\right)  \Leftrightarrow
\lambda,\overline{\lambda }\in \lambda \left(  A\overline{A}\right)
.$

Let $s$ be an arbitrary eigenvalue of $A\overline{A}.$ Then $\left
\vert s\right \vert \leq1.$ Consider three cases. Case 1: $1\geq
s\geq0.$ Then it follows that $\lambda=\sqrt{s}\in
\mathrm{co}\lambda \left(  A\right)  $ and hence $\left \vert
\lambda \right \vert \leq1.$ Case 2: $-1\leq s<0.$ Then we see that
$\lambda=s\in \mathrm{co}\lambda \left(  A\right)  $ and hence
$\left \vert \lambda \right \vert \leq1.$ Case 3:
$\operatorname{Im}\left(  \lambda \right)
\neq0.$ In this case, we see that either $\lambda=s\in \mathrm{co}%
\lambda \left(  A\right)  $ or $\lambda=\overline{s}\in \mathrm{co}%
\lambda \left(  A\right)  .$ In both cases, there holds $\left \vert
\lambda \right \vert =\left \vert s\right \vert \leq1.$ The proof is
completed.

\end{document}